\documentclass[english,letterpaper,11pt,reqno]{amsart}

\usepackage{appendix}
\usepackage{amsbsy}
\usepackage{amsfonts}
\usepackage{amsmath}
\usepackage{amssymb}
\usepackage{amsthm}
\usepackage{graphicx}
\usepackage{ifthen}
\usepackage{textcomp}
\usepackage{soul}
\usepackage{enumitem,kantlipsum}
\usepackage{tikz}
\usepackage{mathrsfs}
\usepackage{cancel}
\usepackage{lmodern}
\usepackage{textcomp}

\usepackage{dsfont}
\usepackage[bookmarksnumbered,colorlinks]{hyperref}
\hypersetup{colorlinks=true, linkcolor=blue, citecolor=blue}
\newcommand{\RR}{\mathbb{R}}

\newtheorem{thm}{Theorem}

\newtheorem{cor}{Corollary}

\newtheorem{defn}{Definition}
\newtheorem{prop}{Proposition}
\newtheorem*{definition-non}{Definition}
\newtheorem*{theorem-non}{Theorem}
\newtheorem*{Proposition-non}{Proposition}
\newtheorem*{lemma-non}{Lemma}
\newtheorem*{corollary-non}{Corollary}

\newcommand{\beqa}{\begin{eqnarray}}
\newcommand{\beq}{\begin{equation}}
\newcommand{\eeqa}{\end{eqnarray}}
\newcommand{\eeq}{\end{equation}}

\newcommand\imp{\hspace{.2in}\Rightarrow\hspace{.2in}}

\newcommand\cd[2]{\nabla_{\!#1}{#2}}

\newcommand\gL{h}
\newcommand\gpw{h_{\scriptscriptstyle \Omega}}

\newcommand\comma{\hspace{.2in},\hspace{.2in}}
\newcommand\commas{\hspace{.1in},\hspace{.1in}}
\makeatletter
\newcommand*{\defeq}{\mathrel{\rlap{%
                     \raisebox{0.24ex}{$\m@th\cdot$}}%
                     \raisebox{-0.24ex}{$\m@th\cdot$}}%
                     =}
\makeatother

\makeatletter
\newcommand*\owedge{\mathpalette\@owedge\relax}
\newcommand*\@owedge[1]{%
  \mathbin{%
    \ooalign{%
      $#1\m@th\bigcirc$\cr
      \hidewidth$#1\m@th\wedge$\hidewidth\cr
    }%
  }%
}
\makeatother

\usepackage{enumitem, color, amssymb}

\begin{document}
\title[]{Almost K\"ahler metrics and pp-wave spacetimes}
\author[]{Amir Babak Aazami and Robert Ream}
\address{Clark University\hfill\break\indent
Worcester, MA 01610}
\email{aaazami@clarku.edu, rream@clarku.edu}

\maketitle
\begin{abstract}
We establish a one-to-one correspondence between a class of strictly almost K\"ahler metrics on the one hand, and Lorentzian pp-wave spacetimes on the other; the latter metrics are well known in general relativity, where they model radiation propagating at the speed of light.  Specifically, we construct families of complete almost K\"ahler metrics by deforming pp-waves via their propagation wave vector.  The almost K\"ahler metrics we obtain exist in all dimensions $2n  \geq 4$, and are defined on both $\RR^{2n}$ and $\mathbb{S}^1\times\mathbb{S}^1 \times M$, where $M$ is any closed almost K\"ahler manifold; they are not warped products, they include noncompact examples with constant negative scalar curvature, and all of them have the property that their fundamental 2-forms are also co-closed with respect to the Lorentzian pp-wave metric.  Finally, we further deepen this relationship between almost K\"ahler and Lorentzian geometry by utilizing Penrose's ``plane wave limit," by which every spacetime has, locally, a pp-wave metric as a limit: using Penrose's construction, we show that in all dimensions $2n \geq 4$, every Lorentzian metric admits, locally, an almost K\"ahler metric of this form as a limit.
\end{abstract}

\section{Introduction}
A Riemannian metric $g$ is \emph{almost K\"ahler} if there is an almost complex structure $J$ compatible with $g$ and if the corresponding 2-form $g(\cdot,J)$ is a symplectic form.  What makes this setting ``almost" K\"ahler is that $J$ itself need not be integrable, as it would be for any K\"ahler metric; i.e., $J$ need not give rise to an atlas of holomorphic coordinate charts.  By omitting complex geometry in this way, almost K\"ahler metrics thereby sit at the boundary between the symplectic and the K\"ahler categories; a very comprehensive survey of them can be found in \cite{apostolov}.  In this paper we shine a light on these metrics from a new direction, by showing that a class of them derive from\,---\,indeed, are in one-to-one correspondence with\,---\,a distinguished class of \emph{Lorentzian} metrics, namely the so called \emph{pp-wave} spacetimes modeling radiation propagating at the speed of light (see, e.g., \cite[Chapter~13]{beem}); among the many remarkable properties exhibited by pp-waves is the fact that many of them have vanishing curvature invariants, and yet are not flat \cite{coley}.   In saying that our almost K\"ahler metrics $g$ ``derive from" them, we mean that if $h$ is a pp-wave metric, then $g$ will be given by
$$
g \defeq h + 2T^{\flat}\otimes T^{\flat}
$$
for a suitable choice of vector field $T$ satisfying $h(T,T) = -1$ (in the parlance of Lorentzian geometry, a so called ``timelike" vector field; here $T^{\flat} = h(T,\cdot)$ is the one-form $h$-metrically equivalent to $T$).  In our setting, $T$ will arise from deforming the propagation wave vector of the radiation being modeled by the pp-wave metric, and we will say that $g$ is ``dual" to the pp-wave metric $h$.  In general, any choice of timelike $T$ will yield a Riemannian metric as above, but as we show in Section \ref{sec:complete}  below, our particular choice of $T$ will always yield, not only an almost K\"ahler metric, but a \emph{complete} one (our metrics are defined in both the compact and the noncompact settings, and therefore completeness must be explicitly verified in the latter).
\vskip 3pt
In the remainder of this Introduction, we outline how our almost K\"ahler metrics compare with those already in the literature.  To the best of our knowledge, the first noteworthy example of a strictly almost K\"ahler metric was exhibited by \cite{thurston}, a compact 4-manifold defined as a 2-torus bundle over a 2-torus; noncompact examples\,---\,other than the tangent bundles of certain Riemannian manifolds, which were known\,---\,came soon afterwards in \cite{watson}, which also contained compact examples by way of Thurston's.   Thurston's example was generalized to higher dimensions in \cite{cordero}. In dimension 4, where the Hodge star operator can be used, \cite{kim} obtained strictly almost K\"ahler metrics via deformations of scalar-flat K\"ahler metrics.  Another class of examples, due to B\'erard-Bergery, can be found on Einstein metrics $M$, with positive Einstein constant, that admit certain Riemannian submersions $P \to M$ with $P$ a principal $\mathbb{S}^1$-bundle; see \cite[Theorem 9.76,~p.~255]{besse}.  Finally, \cite{jelonek} showed the existence of strictly almost K\"ahler metrics on products $\mathbb{S}^1\times \mathbb{S}^1 \times M$ where $M$ is any almost K\"ahler metric, including those with constant negative scalar curvature on the $6$-torus; these metrics are different from our compact examples because those in \cite{jelonek} are warped products, whereas ours are not; furthermore, our construction also yields complete examples on $\RR^{2n}$.  Finally, conformally flat examples of the form $\RR^n \times M$, where $M$ is a (not necessarily complete) Riemannian manifold, were constructed in \cite{catalano}.  In comparison, the distinguishing feature of our examples is their relationship, not merely to pp-waves, but to Lorentzian geometry in general: we show that \emph{every} (even-dimensional) Lorentzian manifold admits, locally, an almost K\"ahler metric in an appropriate limit, namely, the ``plane wave limit" due to Penrose \cite{penrose}.
\vskip 3pt
This paper is organized as follows.  Section \ref{subsec:pp:waves} gives a brief overview of pp-wave spacetimes, emphasizing their geometric and geodesic properties.  Sections \ref{sec:complete} and \ref{sec:curv} introduce the candidates for our almost K\"ahler metrics on $\RR^{2n}$, showing that they are complete and computing their curvature; in Section \ref{sec:Kahler} we demonstrate that they are, in fact, strictly almost K\"ahler metrics (Theorem \ref{thm:almost}), and that they have the added property that their fundamental 2-form is co-closed with respect to the Lorentzian pp-wave metric as well (Corollary \ref{cor:co-closed}).  Section \ref{sec:compact} then generalizes this construction to the compact setting, on manifolds of the form $\mathbb{S}^1 \times \mathbb{S}^1 \times M$, where $M$ is any closed almost K\"ahler manifold (Theorem \ref{thm:torus}).
Finally, Section \ref{sec:Penrose} establishes a deeper connection between almost K\"ahler and Lorentzian geometry, by showing that every (even-dimensional) spacetime admits, locally, an almost K\"ahler metric via Penrose's plane wave limit mentioned above; this is codified in Theorem \ref{thm:3}; we also include here an introduction to Penrose's limit itself.
\section{Brief overview of pp-waves}
\label{subsec:pp:waves}
The class of pp-wave spacetimes have their origin in gravitational physics and have been intensely studied therein; see, e.g., \cite{AMS}, \cite{flores}, and \cite[Chapter 13]{beem}. The definition we give here is due to \cite{globke, leistner}; it is in fact the more modern, coordinate-independent version of the ``standard" definition appearing in the physics literature. Before stating it, recall that with respect to a Lorentzian metric $\gL$ (with index $-\!+\!+\cdots+$), nonzero vectors $X$ divide into three types:
\beqa
\text{$X$ is}~\left\{\begin{array}{ccc}
\text{``spacelike"} &\text{if}& \text{$\gL(X,X) > 0$},\\
\text{``timelike"} &\text{if}& \text{$\gL(X,X) < 0$},\\
\text{``lightlike"} &\text{if}& \text{$\gL(X,X) = 0$}.
\end{array}\right.\nonumber
\eeqa

\begin{defn}[\cite{globke}]
On a \emph{(}compact or noncompact\emph{)} manifold $M$, a Lorentzian metric $\gL$ is a \emph{pp-wave} if it admits a globally defined lightlike vector field $V$ that is parallel, $\nabla V = 0$ \emph{(}$\nabla$ is the Levi-Civita connection of $\gL$\emph{)}, and if its curvature endomorphism $R$ satisfies
\beqa
\label{eqn:ppR}
R(X,Y) = 0~~~\text{for all $X,Y \in \Gamma(V^{\perp})$}.
\eeqa
If in addition $\nabla_{\!X}R  = 0$ for all $X \in \Gamma(V^{\perp})$, then $(M,h)$ is a \emph{plane wave}.  
\end{defn}

Locally, such manifolds always take the following form, a special case of a class of coordinates known as ``Walker coordinates" \cite{walker}:

\begin{theorem-non}[\cite{leistner}]
On any pp-wave $(M,\gL)$, there exist local coordinates $(v,u,x^3,\dots,x^n)$ in which $V = \partial_v$ and
\beqa
\label{eqn:metric}
\gL = H(u,x^3,\dots,x^n)du^2 + 2dvdu + \sum_{i=3}^n (dx^i)^2
\eeqa
for some smooth function $H(u,x^3,\dots,x^n)$ independent of $v$.  Furthermore, $(M,h)$ will be a plane wave if and only if $H$ is a quadratic polynomial in $x^3,\dots,x^n$.
If \eqref{eqn:metric} exists globally on $\RR^n$, then $(\RR^n,h)$ is a \emph{standard pp-wave} or a \emph{standard plane wave}.  The universal cover of a compact pp-wave is globally isometric to a standard pp-wave. 
\end{theorem-non}

Although standard pp-waves on $\RR^n$ are the most common, \emph{compact} pp-waves exist also; e.g., on the $n$-torus $\mathbb{T}^n = \mathbb{S}^1 \times \cdots \times \mathbb{S}^1$, as shown in \cite{leistner}.  The coordinates \eqref{eqn:metric} make pp-waves look deceptively simple, nevertheless they are a remarkable class of Lorentzian metrics. Their defining feature is the parallel lightlike vector field $V$ which, in the local coordinates \eqref{eqn:metric}, is the gradient
$$
V  = \partial_v = \text{grad}\,u.
$$
This vector field models the wave vector of a gravitational or electromagnetic wave propagating at the speed of light.  Indeed, not only are plane waves solutions to the linearized Einstein equations, but, as shown in \cite{Mon}, any vacuum solution to the Einstein equations that possesses a non-homothetic conformal Killing vector field is either conformally flat or a pp-wave.  In fact the ``wave nature" of pp-waves manifests directly as follows: if for an arbitrary lightlike vector field $Z$ one imposes the geodesic condition
$
\cd{Z}{Z} = 0
$
locally in the coordinates \eqref{eqn:metric}, then in dimension 3 this reduces to an (inviscid) Burgers' PDE, which is well known to describe wave motion and shock waves (in higher dimensions it becomes a coupled system of Burger's PDEs).  Their wave nature aside, our interest in pp-waves is due in particular to their curvature and their geodesic properties:
\begin{enumerate}[leftmargin=.2in]
\item[1.] Curvature: For any choice of $H(u,x^3,\dots,x^n)$, the metric \eqref{eqn:metric} is always scalar-flat; in fact, something deeper is true: for certain choices of $H$, all curvature invariants of \eqref{eqn:metric} vanish identically (see \cite{coley}), and yet the metric is not flat in general (a distinctly Lorentzian phenomenon).  It is, however, almost Ricci-flat: in the coordinate basis $\{\partial_v,\partial_u,\partial_3,\dots,\partial_n\}$, the only nonzero component of the Ricci tensor is
$$
\text{Ric}(\partial_u,\partial_u) = -\frac{1}{2}\sum_{i=3}^n H_{ii},
$$
where $H_{ii} \defeq \frac{\partial^2H}{\partial x^i \partial x^i}$.
Thus a pp-wave is Ricci-flat if and only if $H$ is harmonic in $x^3,\dots,x^n$.
\item[2.] Geodesics: In local coordinates \eqref{eqn:metric}, the geodesic equations of motion are $n$ second-order ODEs in $v,u,x^3,\dots,x^n$\,---\,but in fact they reduce to just $n-2$ ODEs, namely 
\beqa
\label{eqn:Hode}
\ddot{x}^i = H_i(t,x^3(t),\dots,x^n(t)) \comma i =x^3,\dots,x^n,
\eeqa
where dots indicate derivatives taken with respect to the affine parameter $t$, which can in fact be scaled to equal the coordinate $u$. In other words, the geodesic equations of motion reduce entirely to a (generally time-dependent) Hamiltonian system\,---\,and the completeness of such systems is well understood; see, e.g., \cite{ebin,weinstein,flores,FS-Ehlers}.  Indeed, as we'll show, \eqref{eqn:Hode} is ultimately the reason why the (Riemannian) almost K\"ahler metrics we construct on $\RR^{2n}$ will be geodesically complete.
\end{enumerate}

\section{From pp-waves to complete Riemannian metrics}
\label{sec:complete}
It is precisely these two features of pp-waves that make them ideally suited to construct distinguished Riemannian metrics, a process which we now describe.  When a nowhere vanishing timelike vector field $T$ is present, any Lorentzian metric $h$ has a Riemannian ``dual" given by
\beqa
\label{eqn:LS0}
g \defeq \gL + 2T^{\flat} \otimes T^{\flat}.
\eeqa
(Here $T^{\flat} = h(T,\cdot)$ is the one-form $h$-metrically equivalent to $T$, and in fact $g(T,\cdot) = -h(T,\cdot)$; note that we are assuming for convenience here that $h(T,T) = -1$.)   The relationship \eqref{eqn:LS0} is well known and has been studied extensively; see, e.g., \cite{olea} for a recent analysis which includes, among other things, curvature formulae.  We start by making clear what will play the role of the vector field ``bridge," $T$, for us: 

\begin{defn}
Let $(M,h)$ be a compact or noncompact pp-wave.  For a given unit timelike vector field $T$ on $(M,h)$, the Riemannian metric
\beqa
\label{def:dual}
g \defeq \gL + 2T^{\flat} \otimes T^{\flat}
\eeqa
is said to be \emph{$T$-dual to} $h$.  On a standard pp-wave $(\RR^n,h)$ in the coordinates \eqref{eqn:metric}, define the unit timelike vector field
\beqa
\label{def:T}
T \defeq \frac{1}{2}\big(H+1\big)\partial_v -  \partial_u.
\eeqa
For this choice of $T$, the corresponding Riemannian metric \eqref{def:dual} will be called the \emph{standard Riemannian dual}.
\end{defn}

An important question is when this dual metric $g$ will be complete.  One general criterion for the completeness of the Riemannian metric $g$ in \eqref{eqn:LS0} is the following: if $g(T,\cdot)$ is bounded on $TM$, then $g$ will be complete if $h$ is (this is a consequence of \cite[Proposition 3.4]{candela}).  In our case, it turns out that a direct analysis of the geodesic equations of \eqref{def:dual}, without recourse to $h$, is better.  Indeed, it turns out the choice of  \eqref{def:T} will
\emph{always} yields a complete $g$\,---\,a testament to the simplicity of the Hamiltonian geodesic equations \eqref{eqn:Hode} of the original pp-wave: 

\begin{prop}
\label{prop:compn}
The standard Riemannian dual $g$ is complete for any choice of smooth $H(u,x^3,\dots,x^n)$. Furthermore, $\partial_v$ is a constant length Killing vector field with respect to $g$.
\end{prop}

\begin{proof}
Setting $i =x^3,\dots,x^n$, and using the coordinate components of $g$ in \eqref{eqn:gn} below, its non-vanishing Christoffel symbols are
\beqa
\Gamma^u_{vi} = -2\Gamma^i_{vu} = H_i \!\!\!\!\!&\comma&\!\!\!\!\!\Gamma^v_{vi} = \Gamma^i_{uu} = -\Gamma^u_{ui} = -\frac{1}{2}HH_i,\nonumber\\
\Gamma^v_{uu} = \frac{1}{2}H_u \!\!\!\!\!&\comma&\!\!\!\!\! \Gamma^v_{ui} = -\frac{1}{4}(H^2-1)H_i,\nonumber
\eeqa
with corresponding geodesic equations of motion
\beqa
\label{eqn:geodn}
\left.\begin{array}{ccc}
\ddot{v} \!\!&=&\!\! \frac{1}{2}(2\dot{v}+\dot{u}H)H \Big(\!\sum_{i=3}^n H_i \dot{x}^i\Big) - \frac{1}{2}\Big(\!\sum_{i=3}^n H_i \dot{x}^i\dot{u}\Big) - \frac{1}{2}H_u\,\dot{u}^2,\phantom{\Big(\Big)}\\
\ddot{u} \!\!&=&\!\! -(2\dot{v} + \dot{u} H) \Big(\!\sum_{i=3}^n H_i \dot{x}^i\Big),\phantom{\Big(\Big)}\\
\ddot{x}^i \!\!&=&\!\! \frac{1}{2} (2\dot{v} + \dot{u} H) H_i\,\dot{u}.\phantom{\Big(\Big)}\\
\end{array}\right\}
\eeqa
The first step is to recognize that $\ddot{v}$ is integrable by using the equation for $\ddot{u}$:
$$
\ddot{v}  = -\frac{1}{2}H\ddot{u} -\frac{1}{2}\bigg(\sum_{i=3}^n H_i \dot{x}^i\dot{u}\bigg)- \frac{1}{2}H_u\,\dot{u}^2 = -\frac{1}{2}\frac{d}{dt}\big(H\dot{u}\big). 
$$
This yields a constant of the motion,
$$
2\dot{v} + \dot{u} H = 2\dot{v}_0 + \dot{u}_0 H_0 \defeq c
$$
(with ``$0$" denoting the initial value), which in turn simplifies \eqref{eqn:geodn}:
\beqa
\left.\begin{array}{ccc}
\ddot{v} \!\!&=&\!\! \frac{c}{2}H \Big(\!\sum_{i=3}^n H_i \dot{x}^i\Big) - \frac{1}{2}\Big(\!\sum_{i=3}^n H_i \dot{x}^i\dot{u}\Big) - \frac{1}{2}H_u\,\dot{u}^2,\phantom{\Big(\Big)}\\
\ddot{u} \!\!&=&\!\! -c\sum_{i=3}^n H_i \dot{x}^i,\phantom{\Big(\Big)}\\
\ddot{x}^i \!\!&=&\!\! \frac{c}{2} H_i\,\dot{u}.\phantom{\Big(\Big)}\\
\end{array}\right\}\nonumber
\eeqa
If $H(u,x^3,\dots,x^n)$ is a given smooth function for which the solutions $x^3(t),\dots,x^n(t), u(t)$ exist for all $t \in \RR$, then so must $v(t)$.  Hence we show the former, by observing that $\ddot{x}^3,\dots,\ddot{x}^n,\ddot{u}$ combine to yield another constant of the motion; indeed, denoting the initial values by the subscript ``$0$" as above,
\beqa
\label{eqn:2}
2\sum_{i=3}^n\ddot{x}^i\dot{x}^i = -\ddot{u}\dot{u} \imp \sum_{i=3}^n(\dot{x}^i(t))^2 + \frac{1}{2}\dot{u}(t)^2 = \underbrace{\sum_{i=3}^n(\dot{x}_0^i(t))^2 + \frac{1}{2}\dot{u}_0^2}_{c_2}.
\eeqa
In particular, each $|\dot{x}^i(t)| \leq \sqrt{c_2}$, in which case,
$$
|x^i(t)|\,-\,|x_0^i| \leq \bigg|\int_{t_0}^t \dot{x}^i(s)\,ds\bigg| \leq \int_{t_0}^t |\dot{x}^i(s)|\,ds \leq \sqrt{c_2}\int_{t_0}^t ds = \sqrt{c_2}(t-t_0),
$$
so that $x^i(t)$ must be bounded over any compact interval $[t_0,t]$; likewise with $u(t)$. In particular, their maximal solutions must be global. Finally, that  $\partial_v$ is a constant length Killing vector field follows because, with respect to the (global) coordinate basis $\{\partial_v,\partial_u,\partial_3,\dots,\partial_n\}$, the $g_{ij}$'s are independent of $v$:
\beqa
\label{eqn:gn}
(g_{ij}) = \begin{pmatrix}2 & H& 0 & 0 &\cdots &0\\ H & \frac{1}{2}(1+H^2)& 0& 0 &\cdots & 0 \\ 0 & 0 &1 &0 & \cdots & 0\\ 0 &0&0&1&\cdots &0\\\vdots & \vdots & \vdots & \vdots & \ddots & \vdots\\0 &0&0&0&\cdots & 1\end{pmatrix}\cdot
\eeqa
This completes the proof.
\end{proof}

\section{The curvature of $g$}
\label{sec:curv}

Note that \eqref{eqn:gn} is not a warped product, though it is an example of so called \emph{semigeodesic coordinates} in Riemannian geometry; see, e.g., \cite{Lee}.  Note also that $\partial_v$ is not parallel with respect to $g$, as it was with respect to $h$.
So much for the geodesic completeness of these metrics; we now move on to computing their curvature; to do so, it is best to work, not in the coordinate basis $\{\partial_v,\partial_u,\partial_{x^3},\dots,\partial_{x^n}\}$, but rather with respect to the following globally defined $g$-orthonormal frame:

\begin{prop}
\label{prop:Rm}
With respect to the orthonormal frame $\{T,X_3,\dots,X_{n},Z\}$ defined by
\beqa
\label{eqn:Tn}
T \defeq \frac{1}{2}(H+1)\partial_v-\partial_u \commas X_i \defeq \partial_i \commas Z \defeq \frac{1}{2}(H-1)\partial_v-\partial_u,
\eeqa
any standard Riemannian dual $g$ in Proposition \ref{prop:compn} has Ricci tensor
\beqa
\label{eqn:Ric0}
\emph{Ric}_g = \frac{1}{2}\begin{pmatrix}\!\!\!\!\!\!\!\phantom{\Big(\Big)}\sum_{i=3}^{n}H_{ii}&H_{u3} & H_{u4} &\cdots & H_{un}&-\!\sum_{i=3}^{n}H_{i}^2\\\phantom{\Big(\Big)}&-H_3^2&-H_{3}H_4 & \cdots&-H_3H_{n}&-H_{3u}\\ \phantom{\Big(\Big)}&&-H_{4}^2 & \cdots&-H_4H_n&-H_{4u}\\ \phantom{\Big(\Big)}&&&\ddots&\vdots&\vdots\\ \phantom{\Big(\Big)}&&&&-H_n^2&-H_{nu}\\\phantom{\Big(\Big)}& & && &-\!\sum_{i=3}^{n}H_{ii}\end{pmatrix}\cdot
\eeqa
\end{prop}

\begin{proof}
Using the $g$-orthonormal frame \eqref{eqn:Tn}, we compute the diagonal components of the Ricci tensor.  Using the Christoffel symbols computed in Proposition \ref{prop:compn}, the covariant derivatives are
\beqa
\cd{T}{T} = -\cd{Z}{Z} = \sum_{i=3}^{n}\frac{H_i}{2}X_i \!\!\!\!\!&\comma&\!\!\!\!\! \cd{T}{Z} = \cd{Z}{T} = \cd{X_i}{X_j} = 0,\nonumber\\
\cd{T}{X_i} = -\cd{X_i}{Z} = -\frac{H_i}{2}T \!\!\!\!\!&\comma&\!\!\!\!\! \cd{Z}{X_i} = -\cd{X_i}{T}  = \frac{H_i}{2}Z.\nonumber
\eeqa
The components of the Riemann 4-tensor $\text{Rm}_g$ are now easily computed:
\beqa
\left\{\begin{array}{lcl}
\text{Rm}_g(Z,T,T,Z) \!\!&=&\!\! g(\cd{Z}{\cd{T}{T}}-\cancel{\cd{T}{\cd{Z}{T}}}-\cancel{\cd{[Z,T]}{T}},Z) = \sum_{i=3}^{2n}\Big(\!\frac{H_{i}}{2}\!\Big)^{\!2},\phantom{\Big(\Big)}\nonumber\\
\text{Rm}_g(Z,T,X_i,Z) \!\!&=&\!\! g(\cancel{\cd{Z}{\cd{T}{X_i}}}-\cd{T}{\cd{Z}{X_i}}-\cancel{\cd{[Z,T]}{X_i}},Z) = \frac{H_{iu}}{2},\phantom{\Big(\Big)}\nonumber\\
\text{Rm}_g(T,Z,X_i,T) \!\!&=&\!\! g(\cancel{\cd{T}{\cd{Z}{X_i}}}-\cd{Z}{\cd{T}{X_i}}-\cancel{\cd{[T,Z]}{X_i}},T) = -\frac{H_{iu}}{2},\phantom{\Big(\Big)}\nonumber\\
\text{Rm}_g(X_i,T,Z,X_j) \!\!&=&\!\! g(\cancel{\cd{X_i}{\cd{T}{Z}}}-\cd{T}{\cd{X_i}{Z}}-\cd{[X_i,T]}{Z},X_j) = -\frac{H_{i}H_j}{2},\phantom{\Big(\Big)}\nonumber\\
\text{Rm}_g(T,X_i,X_j,T) \!\!&=&\!\! g(\cancel{\cd{T}{\cd{X_i}{X_j}}}-\cd{X_i}{\cd{T}{X_j}}-\cd{[T,X_i]}{X_j},T) = \frac{H_{ij}}{2}-\frac{H_iH_j}{4},\phantom{\Big(\Big)}\nonumber\\
\text{Rm}_g(Z,X_i,X_j,Z) \!\!&=&\!\! g(\cancel{\cd{Z}{\cd{X_i}{X_j}}}-\cd{X_i}{\cd{Z}{X_j}}-\cd{[Z,X_i]}{X_j},Z)  = -\frac{H_{ij}}{2} -\frac{H_iH_j}{4},\phantom{\Big(\Big)}\nonumber\\
\text{Rm}_g(X_j,T,X_i,X_j) \!\!&=&\!\! g(\cancel{\cd{X_j}{\cd{T}{X_i}}}-\cancel{\cd{T}{\cd{X_j}{X_i}}}-\cancel{\cd{[X_j,T]}{X_i}},X_j) = 0,\phantom{\Big(\Big)}\nonumber\\
\text{Rm}_g(X_j,Z,X_i,X_j) \!\!&=&\!\! g(\cancel{\cd{X_j}{\cd{Z}{X_i}}}-\cancel{\cd{Z}{\cd{X_j}{X_i}}}-\cancel{\cd{[X_j,Z]}{X_i}},X_j) = 0,\phantom{\Big(\Big)}\nonumber\\
\text{Rm}_g(T,Z,X_i,X_j) \!\!&=&\!\! g(\cancel{\cd{T}{\cd{Z}{X_i}}}-\cancel{\cd{Z}{\cd{T}{X_i}}}-\cancel{\cd{[T,Z]}{X_i}},X_j) = 0,\phantom{\Big(\Big)}\nonumber\\
\text{Rm}_g(X_k,X_i,X_j,X_k) \!\!&=&\!\! g(\cancel{\cd{X_k}{\cd{X_i}{X_j}}}-\cancel{\cd{X_i}{\cd{X_k}{X_j}}}-\cancel{\cd{[X_k,X_i]}{X_j}},X_k) = 0,\phantom{\Big(\Big)}\nonumber\\
\text{Rm}_g(T,X_i,X_j,X_k) \!\!&=&\!\! g(\cancel{\cd{T}{\cd{X_i}{X_j}}}-\cancel{\cd{X_i}{\cd{T}{X_j}}}-\cancel{\cd{[T,X_i]}{X_j}},X_k) = 0,\phantom{\Big(\Big)}\nonumber\\
\text{Rm}_g(Z,X_i,X_j,X_k) \!\!&=&\!\! g(\cancel{\cd{Z}{\cd{X_i}{X_j}}}-\cancel{\cd{X_i}{\cd{Z}{X_j}}}-\cancel{\cd{[Z,X_i]}{X_j}},X_k) = 0.\phantom{\Big(\Big)}\nonumber
\end{array}\right.
\eeqa
From this data, the Ricci tensor \eqref{eqn:Ric0} easily follows.
\end{proof}

These metrics are strongly controlled by their scalar curvature, somewhat reminiscent of the behavior of anti-self-dual 4-manifolds:

\begin{cor}
\label{cor:flat}
Let $(\RR^n,g)$ be the standard Riemannian dual to a pp-wave metric. Then the scalar curvature $\emph{scal}_g$ is nonpositive and vanishes if and only if $g$ is flat.
\end{cor}

\begin{proof}
Since the frame \eqref{eqn:Tn} is orthonormal, the scalar curvature is the trace of the matrix \eqref{eqn:Ric0}:
\beqa
\label{eqn:nscalar}
\text{scal}_g = -\frac{1}{2}\sum_{i=3}^n H_{i}^2,
\eeqa
which vanishes if and only if each $H_i \defeq \frac{\partial H}{\partial x^i} = 0$.  As Proposition \ref{prop:Rm} shows, this is the case if and only if $\text{Rm}_g =  0$.
\end{proof}
The property ``$\text{scal}_g \leq 0$ and flat if $\text{scal}_g = 0$" in Corollary \ref{cor:flat} bears comparison to the following two facts from almost K\"ahler geometry: i) anti-self-dual almost K\"ahler metrics on oriented 4-manifolds must satisfy $\text{scal}_g \leq 0$, and are K\"ahler if and only if $\text{scal}_g = 0$; ii) in any even dimension, this is also true for any conformally flat almost K\"ahler metric (see \cite[Proposition~1]{apostolov}).  As Theorem \ref{thm:almost} below shows, our almost K\"ahler metrics have the same property.

\section{Strictly almost K\"ahler metrics on $\RR^{2n}$}
\label{sec:Kahler}
First, recall that an \emph{almost complex structure} $J$ is a smooth endomorphism of the tangent bundle satisfying $J^2 = -1$; it is \emph{compatible} with a Riemannian metric $g$ if $g(J,J) = g$.  Given such a pair, the 2-form $\omega \defeq g(\cdot,J)$ is called the \emph{fundamental 2-form}.

\begin{defn}
A $2n$-dimensional Riemannian manifold $(M,g)$ is an \emph{almost K\"ahler manifold} if there is an almost complex structure $J$ compatible with $g$ and such that the fundamental 2-form $\omega \defeq g(\cdot,J)$ is closed.
\end{defn}

Note that if $J$ were in addition \emph{integrable}, meaning that its \emph{Nijenhuis tensor}
$$
N_{\scriptscriptstyle J}(a,b) \defeq [Ja,Jb] -  J[Ja,b] - J[a,Jb] - [a,b]
$$
vanished identically, then $(g,J)$ would be a \emph{K\"ahler manifold}.  (If $\omega$ is not necessarily closed, then $(g,J)$ is an \emph{almost Hermitian manifold}.) In any case, we are now going to use Proposition \ref{prop:compn} to construct explicit examples of complete almost K\"ahler metrics in all even dimensions $\geq 4$. Our almost complex structure $J$ will be defined with respect to the globally defined $g$-orthonormal basis $\{T,X_3,\dots,X_{2n},Z\}$ defined in \eqref{eqn:Tn} above:
\beqa
\label{eqn:Jn}
JT \defeq Z \comma JX_3 \defeq X_4,\comma\dots\comma JX_{2n-1} \defeq X_{2n}.
\eeqa
This choice of $J$ yields many complete almost K\"ahler metrics, in both the compact and the noncompact setting; we begin with the latter:

\begin{thm}
\label{thm:almost}
Let $(\RR^{2n},g)$ be the standard Riemannian dual \eqref{eqn:gn} to a pp-wave metric, with $n\geq 2$.  Let $J$ be the almost complex structure defined via \eqref{eqn:Jn}. Then $(\RR^{2n},g,J)$ is a complete almost K\"ahler manifold with nonpositive scalar curvature, which is K\"ahler \emph{(}\hspace{-.01in}in fact, flat\emph{)} if and only if $g$ is scalar flat.
\end{thm}

\begin{proof}
That $J$ is compatible with $g$ follows by definition of $J$ in \eqref{eqn:Jn}, together with the fact that $\{T,X_3,\dots,X_{2n},Z\}$ is a $g$-orthonormal basis.  To check that the fundamental 2-form $\omega = g(\cdot,J)$ is closed, we consider the dual basis $\{\tau,\theta^3,\dots,\theta^{2n},\zeta\}$, where
$$
\tau \defeq \frac{1}{2}(H-1)du+dv \comma \theta^i \defeq  dx^i \comma \zeta \defeq -\frac{1}{2}(H+1)du-dv.
$$
With respect to this dual basis,
$$
\omega = \zeta\wedge\tau + \theta^4\wedge\theta^3+\dots +\theta^{2n}\wedge\theta^{2n-1},
$$
which is the standard symplectic form on $\RR^{2n}$.
Thus $\omega$ is closed and $(g,J)$ is an almost K\"ahler metric.  Finally, we determine when $g$ can be K\"ahler; i.e., when $J$ will be integrable. Given $JX_i = X_{i+1}$,
\beqa
N_{\scriptscriptstyle J}(T,X_i) \!\!&=&\!\! \underbrace{[JT,JX_i]}_{-\frac{H_{i+1}}{2}(T-Z)} -  \underbrace{J[JT,X_i]}_{-\frac{H_{i}}{2}J(T-Z)} - \!\!\underbrace{J[T,JX_{i}]}_{-\frac{H_{i+1}}{2}J(T-Z)}\! -\! \underbrace{[T,X_i]}_{-\frac{H_i}{2}(T-Z)}\nonumber\\
&=&\!\!  \frac{1}{2}(H_i-H_{i+1})(T-Z) + \frac{1}{2}(H_i+H_{i+1})J(T-Z)\nonumber\\
&=&\!\! 0 \iff H_i=H_{i+1} = 0.\label{eqn:Kahler00}
\eeqa
But as shown in Corollary \ref{cor:flat}, each $H_i = 0$ if and only if $g$ is scalar flat, which is the case if and only if $g$ is flat.
\end{proof}


\begin{cor}
\label{cor:Rm}
Among the strictly almost K\"ahler metrics on $\RR^{2n}$ in Theorem \ref{thm:almost}, there are left-invariant metrics \emph{(}for the additive Lie group structure on $\RR^{2n}$\emph{)}, namely, for any
$$
H(u,x^3,\dots,x^{2n}) \defeq \varphi(u) + \sum_{i=3}^{2n}a_ix^i,
$$
where $\varphi$ is any smooth function on $\RR$ and where each $a_i \in \RR$.
\end{cor}

\begin{proof}
This is a direct consequence of \eqref{eqn:nscalar} in Corollary \ref{cor:flat}, and the fact that in this case the Lie brackets of the $g$-orthonormal basis $\{T,X_3,\dots,X_{2n},Z\}$,
\beqa
\label{eqn:LB00}
[T,X_i] = [Z,X_i] = -\frac{H_i}{2}(T-Z) \comma [T,Z] = [X_i,X_j] = 0,
\eeqa
will have constant structure constants.
\end{proof}

Recall the \emph{codifferential} $\delta$ of the exterior derivative $d$; when applied to the fundamental 2-form $\omega$, it yields a 1-form $\delta\omega$ given by
$$
\delta\omega \defeq (*^{-1} \cdot d \cdot *)\omega,
$$
where $*$ is the Hodge star operator with respect to $g$. The codifferential $\delta$ is the $g$-adjoint of $d$, in the sense that $\int g(\delta \omega_1,\omega_2)\text{dvol} = \int g(\omega_1,d\omega_2)\text{dvol}$.  Its action can equivalently be expressed as
\beqa
\label{eqn:coclosed}
\delta\omega(\cdot) = -\sum_{i}(\cd{E_i}{\omega})(E_i,\cdot),
\eeqa
where $\nabla$ is the Levi-Civita connection of $g$ and where $\{E_1,\dots,E_{2n}\}$ is any $g$-orthonormal basis (see, e.g., \cite[p.~335]{Petersen}). We say that $\omega$ is \emph{co-closed} if $\delta \omega = 0$, a condition that holds automatically for any almost K\"ahler metric $g$.  What is interesting about the fundamental 2-forms of the almost K\"ahler metrics of Theorem \ref{thm:almost} is that they are co-closed, not just with respect to $g$, but also with respect to the \emph{Lorentzian} pp-wave metric $h$:

\begin{cor}
\label{cor:co-closed}
The fundamental 2-forms of the almost K\"ahler metrics of Theorem \ref{thm:almost} are also co-closed with respect to the Lorentzian pp-wave metrics \eqref{eqn:metric} to which they are dual.
\end{cor}

\begin{proof}
The basis $\{T,X_3,\dots,X_{2n},Z\}$ given by \eqref{eqn:Tn} is also $h$-orthonormal, with $h(T,T) = -1$ the timelike direction.  
It is straightforward to show that the Hodge-duals with respect to $g$ and $h$ differ by a constant multiple of $dx^3\wedge\dots\wedge dx^{2n}$, so that being
co-closed with respect to $g$ implies being co-closed with respect to $h$.  (One can also verify this via \eqref{eqn:coclosed}, with the Levi-Civita connection $\nabla^{\scalebox{0.5}{\emph{h}}}$ of $h$ being used in place of $\nabla$.)
\end{proof}

\section{Compact strictly almost K\"ahler metrics}
\label{sec:compact}
We now move to the compact setting, and show that our construction yields a large class of strictly almost K\"ahler metrics here as well; furthermore, since completeness comes ``for free" here, \emph{we can considerably expand our collection of almost K\"ahler metrics by generalizing our notion of a pp-wave}.  The generalization we need is well known in the literature, and consists of allowing the Euclidean ``plane front" $\sum_{i=3}^n (dx^i)^2$ in \eqref{eqn:metric} to be an arbitrary Riemannian metric; see \cite{candela2003} for a detailed study of such metrics. Our definition here is specifically tailored to the compact setting:

\begin{defn}[General plane-fronted waves; \cite{candela2003}]
Let $\varphi,\theta$ denote the standard angular coordinates on $\mathbb{S}^1 \times \mathbb{S}^1$, and let $(M,g_{\scalebox{0.5}{R}})$ be any closed Riemannian manifold.  Let $H$ be an arbitrary smooth function on $\mathbb{S}^1 \times M$\emph{;} i.e., one that is independent of the first angular coordinate $\varphi$.  Then the Lorentzian metric $h$ defined on $\mathbb{S}^1 \times \mathbb{S}^1 \times M$ by
\beqa
\label{eqn:genppwave}
h \defeq 2d\varphi d\theta + Hd\theta^2 + g_{\scalebox{0.5}{R}}
\eeqa
is a \emph{(compact) general plane-fronted wave}.
\end{defn}

Note that we are not requiring the curvature condition \eqref{eqn:ppR} to hold here, though certainly this can be arranged; e.g., by taking $M = \mathbb{S}^1 \times \cdots \times \mathbb{S}^1$ with its standard flat metric (see, e.g., \cite[Example 1]{leistner}).  Although the metrics \eqref{eqn:genppwave} are more general than pp-waves, observe that they still come equipped with a parallel lightlike vector field, namely, $\partial_\varphi = \text{grad}\,\theta$.  In any case, almost K\"ahler metrics exist naturally on such spaces:

%
%

\begin{thm}
\label{thm:torus}
Let $(M,g_{\scalebox{0.5}{R}})$ be a closed almost K\"ahler manifold and consider the Lorentzian metric $h$ \eqref{eqn:genppwave} on $\mathbb{S}^1 \times \mathbb{S}^1 \times M$, with $H$ any smooth function on $\mathbb{S}^1 \times M$\emph{;} i.e., one that is independent of the first angular coordinate $\varphi$.  With respect to the vector field
$$
T \defeq \frac{1}{2}(H+1)\partial_\varphi-\partial_\theta,
$$
let $g$ denote the Riemannian metric dual to $h$\emph{:}
\beqa
\label{eqn:CSC}
g \defeq \gL + 2T^{\flat} \otimes T^{\flat}.
\eeqa
Then $g$ is an almost K\"ahler metric on $\mathbb{S}^1 \times \mathbb{S}^1 \times M$, which is not a warped product, and which is K\"ahler if and only if $H$ is constant on $M$ and $(M,g_{\scalebox{0.5}{R}})$ is K\"ahler. Furthermore, the fundamental 2-form of $g$ is also co-closed with respect to the Lorentzian metric $h$.
\end{thm}

\begin{proof}
Let $J_{\scalebox{0.5}{\emph{R}}}$ be the almost complex structure on $M$ compatible with $g_{\scalebox{0.5}{\emph{R}}}$; let $\{X_1,\dots,X_n,X_{n+1},\dots,X_{2n}\}$ be any locally defined $g_{\scalebox{0.5}{\emph{R}}}$-orthonormal frame on $M$ satisfying
$$
J_{\scalebox{0.5}{\emph{R}}}X_i = X_{n+i} \comma i=1,\dots,n.
$$
Then $J_{\scalebox{0.5}{\emph{R}}}$ extends naturally to an $g$-compatible almost complex structure $J$ on $\mathbb{S}^1 \times \mathbb{S}^1 \times M$ provided we define, as we did on $\RR^{2n}$,
$$
JT \defeq Z \comma Z \defeq \frac{1}{2}(H-1)\partial_\varphi-\partial_\theta.
$$
The analogue of \eqref{eqn:Tn} on $\mathbb{S}^1 \times \mathbb{S}^1 \times M$ is therefore the $g$-orthonormal frame $\{T,X_1,\dots,X_n,X_{n+1},\dots,X_{2n},Z\}$, whose Lie brackets are, analogously to \eqref{eqn:LB00},
\beqa
\label{eqn:LB11}
{[T,X_i]} = {[Z,X_i]} = -\frac{X_i(H)}{2}(T-Z) \comma {[T,Z]} = 0.
\eeqa
To show that the fundamental 2-form $\omega \defeq g(\cdot,J)$ is closed, we may proceed as we did in the non-compact case in Theorem \ref{thm:almost}, by observing that $\omega$ arises from the symplectic structures on $\mathbb{S}^1 \times \mathbb{S}^1$ and $M$, as
$$
\omega = d\varphi \wedge d\vartheta + \omega_{\scalebox{0.5}{\emph{M}}},
$$
where $\omega_{\scalebox{0.5}{\emph{M}}} \defeq g_{\scalebox{0.5}{\emph{R}}}(\cdot,J_{\scalebox{0.5}{\emph{R}}})$; this makes it clear that $(g,J)$ will be almost K\"ahler if and only if $(g_{\scalebox{0.5}{\emph{R}}},J_{\scalebox{0.5}{\emph{R}}})$ itself is so. (Alternatively, one may may verify that $d\omega = 0$ component-by-component, using the Lie brackets \eqref{eqn:LB11}.)  Finally, $g$ will be K\"ahler if and only if $N_{\scriptscriptstyle J}(T,X_i) =  0$ and $N_{\scriptscriptstyle J}(X_i,X_j) =  0$; the former occurs when $X_i(H) = 0$ as in  \eqref{eqn:Kahler00}, so that $H$ must be constant on $M$; the latter occurs when $(M,g_{\scalebox{0.5}{\emph{R}}})$ itself is K\"ahler.  Finally, that $\omega$ is also co-closed with respect to the Lorentzian metric $h$ follows as it did in Corollary \ref{cor:co-closed}, namely, the Hodge-duals with respect to $g$ and $h$ differ by a constant multiple of $\omega_{\scalebox{0.5}{\emph{M}}}$, so that being co-closed with respect to $g$ implies being co-closed with respect to $h$.
\end{proof}

We close with a three final remarks regarding our construction.
\begin{enumerate}[leftmargin=*]
\item[1.] Taking $M$ to be the $2n$-torus, with $g_{\scalebox{0.5}{\emph{R}}}$ its standard flat metric, furnishes a plane-fronted wave \eqref{eqn:genppwave} (in fact, a pp-wave) and thus an almost K\"ahler metric \eqref{eqn:CSC} on the $(2n+2)$-torus.  As this holds for every choice of smooth function $H$ on $\mathbb{S}^1 \times M$, we thus have uncountably many compact examples of almost K\"ahler metrics dual to pp-waves.  

\item[2.] If the product $\mathbb{S}^1\times \mathbb{S}^1 \times M$ has any non-even odd Betti numbers (e.g., for $M = \mathbb{S}^1 \times \mathbb{S}^3$), then not only will $g$ not be K\"ahler, but the manifold $\mathbb{S}^1\times \mathbb{S}^1 \times M$ cannot admit any K\"ahler metric (see, e.g. \cite[p.~84]{besse}).

\item[3.] Speaking of K\"ahler metrics, we can also realize our almost K\"ahler metrics as arising from deformations of them.  This is done by generalizing our definition of pp-wave by allowing $H$ to be  a function of $v$ (for $\RR^{2n}$) or $\varphi$ (on $\mathbb{S}^1 \times \mathbb{S}^1 \times M$) as well.  Then the following is true:
\begin{enumerate}[leftmargin=*]
\item[i.] The vector fields $\partial_v$ or $\partial_\varphi$ would no longer be Killing vector fields. One consequence of this is that the metric on $\RR^{2n}$ may no longer be complete (Proposition \ref{prop:compn} no longer applies). Let us therefore consider only the compact case, $\mathbb{S}^1 \times \mathbb{S}^1 \times M$.
\item[ii.] It turns out that, even when $H$ is taken to be a function of all the coordinates $(\varphi,\theta,x^1,\dots,x^{2n})$, the corresponding Riemannian dual is still  an almost K\"ahler metric.  Indeed, the only change that arises is to the Lie bracket $[T,Z]$ in \eqref{eqn:LB00}, which now becomes
$$
[T,Z] = -\frac{H_\varphi}{2}(T-Z).
$$
But a ``component-by-component" inspection of $d\omega$ reveals that $d\omega=0$ will still hold, as before.  
\item[iii.] Thus our family of strictly almost K\"ahler metrics is enlarged considerably by allowing $H$ to be an arbitrary smooth function defined on $\mathbb{S}^1 \times \mathbb{S}^1 \times M$. But this opens up a new possibility: there are now choices of $H$ for which the Riemannian dual will in fact be a non-flat \emph{K\"ahler} metric. Indeed, the condition for the integrability of $J$, \eqref{eqn:Kahler00}, still holds as before: we must have each $X_i(H) = 0$, for $i=1,\dots,2n$. If we suppose this happens, then we are left with a function on $\mathbb{S}^1 \times \mathbb{S}^1 \times M$ of the form $H=H(\varphi,\theta)$ (as opposed to just $H=H(\theta)$, as before).  But when $H=H(\varphi,\theta)$, the resulting metric on $\mathbb{S}^1 \times \mathbb{S}^1 \times M$ will split as a product of a generally non-flat K\"ahler surface $\mathbb{S}^1 \times \mathbb{S}^1$ and an almost K\"ahler metric on $M$ (compare \eqref{eqn:gn}).  We may therefore view our construction as also arising from a deformation of K\"ahler products.  
\end{enumerate}
\end{enumerate}


\section{Plane wave limits and almost K\"ahler geometry}
\label{sec:Penrose}
In this section we demonstrate how every even-dimensional Lorentzian metric admits, locally, a (Riemannian) strictly almost K\"ahler metric via an appropriate limit.  This limit is the ``plane wave limit" due to Penrose \cite{penrose}, itself an instance of a more general notion of the ``limit of a spacetime" pioneered by Geroch \cite{gerochL}.  In fact the existence of Penrose's limit is intimately connected with the fact, mentioned in Section \ref{subsec:pp:waves} above, that all the curvature invariants of pp-waves vanish.  We begin with a brief, self-contained presentation of Penrose's construction, followed afterwards by Theorem \ref{thm:3}, which connects this limit to the  existence of almost K\"ahler metrics.  Penrose's construction is as follows:  
\begin{enumerate}[leftmargin=*]
\item[1.] Given a Lorentzian metric $h$ and a lightlike gradient vector field $N$,
$$
N \neq 0 \comma h(N,N) = 0 \comma N = \text{grad}_{\scriptscriptstyle h}f,
$$
a so called ``lightlike coordinate system" $(x^0,x^1,x^2,\dots,x^n)$ can be set up with respect to which $N = \frac{\partial}{\partial x^0}$ and such that $h$ has the form
\beqa
\label{eqn:lor*}
(h_{ij}) \defeq
    \begin{pmatrix}
        0 & 1  & 0 & 0 & \cdots & 0\\
        1 & h_{11} & h_{12} & h_{13} & \cdots & h_{1n}\\
        0 & h_{21} &  h_{22} & h_{23} & \cdots &  h_{2n}\\
        0 & h_{31} &  h_{32} & h_{33} & \cdots & h_{3n}\\
        \vdots & \vdots & \vdots & \vdots & \ddots &  \vdots\\
        0 & h_{n1} & h_{n2} & h_{n3} & \cdots & h_{nn}
      \end{pmatrix}\cdot\nonumber
\eeqa
See, e.g., \cite[Proposition 7.14,~p.~61]{pen} for a derivation. (Note that such an $N$ always exists locally, since the eikonal equation $h^{ij}f_if_j = 0$ always admits nontrivial solutions locally.)

\item[2.] Now scale these coordinates by defining another coordinate system $(\tilde{x}^0,\tilde{x}^1,\tilde{x}^2,\dots,\tilde{x}^n)$ via the diffeomorphism $\varphi$ given by
\beqa
\label{eqn:tilde}
(x^0,x^1,x^2,\dots,x^n) \overset{\varphi}{\mapsto} \underbrace{\,(x^0,\Omega^{-2}x^1,\Omega^{-1} x^2,\dots,\Omega^{-1} x^n)\,}_{``(\tilde{x}^0,\tilde{x}^1,\tilde{x}^2,\dots,\tilde{x}^n)"},
\eeqa
where $\Omega > 0$ is a constant.

\item[3.] Next, define a \emph{new} metric $\gpw$ in the new coordinates $(\tilde{x}^0,\tilde{x}^1,\tilde{x}^2,\dots,\tilde{x}^n)$ as follows,
\beqa
\label{newmetric}
\big((\gpw)_{ij}\big) \defeq 
    \underbrace{\,\begin{pmatrix}
        0 & 1  & 0 & 0 & \cdots & 0\\
        1 & \Omega^2h_{11} & \Omega h_{12} & \Omega h_{13} & \cdots & \Omega h_{1n}\\
        0 & \Omega h_{21} &  h_{22} & h_{23} & \cdots & h_{2n}\\
        0 & \Omega h_{31} &  h_{32} & h_{33} & \cdots & h_{3n}\\
        \vdots & \vdots & \vdots & \vdots & \ddots &  \vdots\\
        0 & \Omega h_{n1} & h_{n2} & h_{n3} & \cdots & h_{nn}
      \end{pmatrix}\,}_{\text{defined in the coordinates}~(\tilde{x}^0,\tilde{x}^1,\tilde{x}^2,\dots,\tilde{x}^n)},\nonumber
\eeqa
where each component $(\gpw)_{ij}$ is defined as follows,
\beqa
(\gpw)_{11}(\tilde{x}^0,\tilde{x}^1,\tilde{x}^2,\dots,\tilde{x}^n) \!\!&\defeq&\!\! \Omega^2 \underbrace{\,{h}_{11}(\tilde{x}^0,\Omega^2\tilde{x}^1,\Omega\,\tilde{x}^2,\dots,\Omega\,\tilde{x}^n)\,}_{=\,{h}_{11}(x^0,x^1,x^2,\dots,x^n)}\label{newcomp},\\
(\gpw)_{22}(\tilde{x}^0,\tilde{x}^1,\tilde{x}^2,\dots,\tilde{x}^n) \!\!&\defeq&\!\! \underbrace{\,h_{22}(\tilde{x}^0,\Omega^2\tilde{x}^1,\Omega\,\tilde{x}^2,\dots,\Omega\,\tilde{x}^n)\,}_{=\,{h}_{22}(x^0,x^1,x^2,\dots,x^n)}\label{newcomp2},
\eeqa
and similarly with the others.  Note that as $\Omega \to 0$,
\beqa
\lim_{\Omega \to 0}  (\gpw)_{11} \!\!&\overset{\eqref{newcomp}}{=}&\!\! 0\cdot {h}_{11}(\tilde{x}^0,0,0,\dots,0) = 0,\nonumber\\
\lim_{\Omega \to 0}  (\gpw)_{22} \!\!&\overset{\eqref{newcomp2}}{=}&\!\! {h}_{22}(\tilde{x}^0,0,0,\dots,0),\label{newcomp4}\nonumber
\eeqa
etc.  The crucial fact is that the metric $\gpw$ is \emph{conformal} to the pullback metric $(\varphi^{-1})^*h$; to see this, use the fact that 
$$
((\varphi^{-1})^*h)(\partial_{\tilde{x}^i},\partial_{\tilde{x}^j})d\tilde{x}^i \otimes d\tilde{x}^j = h(\partial_{x^i},\partial_{x^j})dx^i\otimes dx^j,
$$
as well as \eqref{eqn:tilde}, to obtain
\beqa
dx^0 \otimes dx^1 \!\!&=&\!\! \Omega^{2}\,d\tilde{x}^0 \otimes d\tilde{x}^1, \nonumber\\
h_{11}(x^0,x^1,x^2,\dots,x^n)\,dx^1 \otimes dx^1 \!\!&\overset{\eqref{newcomp}}{=}&\!\! \Omega^{2}\,(\gpw)_{11}(\tilde{x}^0,\tilde{x}^1,\tilde{x}^2,\dots,\tilde{x}^n)\,d\tilde{x}^1 \otimes d\tilde{x}^1,\nonumber\\
h_{12}(x^0,x^1,x^2,\dots,x^n)\,dx^1 \otimes dx^2 \!\!&=&\!\! \Omega^{2}\,(\gpw)_{12}(\tilde{x}^0,\tilde{x}^1,\tilde{x}^2,\dots,\tilde{x}^n)\,d\tilde{x}^1 \otimes d\tilde{x}^2,\nonumber\\
h_{22}(x^0,x^1,x^2,\dots,x^n)\,dx^2 \otimes dx^2 \!\!&\overset{\eqref{newcomp2}}{=}&\!\!\Omega^{2}\,(\gpw)_{22}(\tilde{x}^0,\tilde{x}^1,\tilde{x}^2,\dots,\tilde{x}^n)\,d\tilde{x}^2 \otimes d\tilde{x}^2,\nonumber
\eeqa
and so on, which clearly yields the relationship
\beqa
\label{eqn:cm}
(\varphi^{-1})^*h = \Omega^{2}\,\gpw.
\eeqa
In particular, setting $\tilde{h} \defeq (\varphi^{-1})^*h$, the homothety \eqref{eqn:cm} means that the Levi-Civita connections of $\gpw$ and $\tilde{h}$ are equal: $\nabla^{\scriptscriptstyle \Omega} = \nabla^{\scriptscriptstyle \tilde{h}}$.

\item[4.] Finally, take the limit
$$
h_{\scalebox{0.4}{PW}} \defeq \lim_{\Omega \to 0} \gpw = \lim_{\Omega \to 0} \frac{(\varphi^{-1})^*h}{\Omega^{2}}\cdot
$$
This limit metric $h_{\scalebox{0.4}{PW}}$ is precisely
\beqa
\label{plw2}
    \underbrace{\,\begin{pmatrix}
        0 & 1  & 0 & 0 & \cdots & 0\\
        1 & 0 & 0 & 0 & \cdots & 0\\
        0 & 0 &  h_{22}(\tilde{x}^0,0,\dots,0) & h_{23}(\tilde{x}^0,0,\dots,0) & \cdots & h_{2n}(\tilde{x}^0,0,\dots,0)\\
        0 & 0 &  h_{32}(\tilde{x}^0,0,\dots,0) & h_{33}(\tilde{x}^0,0,\dots,0) & \cdots & h_{3n}(\tilde{x}^0,0,\dots,0)\\
        \vdots & \vdots & \vdots & \vdots & \ddots &  \vdots\\
        0 & 0 & h_{n2}(\tilde{x}^0,0,\dots,0) & h_{n3}(\tilde{x}^0,0,\dots,0) & \cdots & h_{nn}(\tilde{x}^0,0,\dots,0)
      \end{pmatrix}\,}_{\text{defined in the coordinates}~(\tilde{x}^0,\tilde{x}^1,\tilde{x}^2,\dots,\tilde{x}^n)},\nonumber
\eeqa
which is in fact a \emph{plane} wave metric in so called ``Rosen coordinates"; i.e., an isometry exists between this metric and \eqref{eqn:metric}, with $H(u,x^3,\dots,x^n)$ quadratic in the $x^i$'s (per the definition of plane wave). For the particulars of this isometry, consult, e.g., \cite{blau2}.  
\end{enumerate}

In Penrose's own words, a neighborhood of the integral curve $\gamma$ of $N = \text{grad}_{\scriptscriptstyle h}f$ through the origin has been expanded ``out to infinity," with the metric homothetically scaled up at the same time, all while keeping $\gamma$ itself unaffected\,---\,effectively ``zooming in" infinitesimally close to $\gamma$.  Although this construction is local and clearly depends on the choice of $N$, certain properties of $h$ are preserved \emph{regardless of how the limit is taken}; e.g., if $h$ is Einstein, then (every) $h_{\scalebox{0.4}{PW}}$ is Ricci flat, and if $h$ is locally conformally flat, then (every) $h_{\scalebox{0.4}{PW}}$ is so as well (see, e.g., \cite{philip}).  Theorem \ref{thm:almost} now provides an easy step to almost K\"ahler geometry:
 \begin{thm}
 \label{thm:3}
 Let $(M,h)$ be a Lorentzian manifold of dimension $2n \geq 4$.  Locally about any point in $M$, take the plane wave limit \emph{$h_{\scalebox{0.4}{PW}}$} of $h$ and express it  in the coordinates \eqref{eqn:metric}.  Then the standard Riemannian dual \eqref{def:dual}-\eqref{def:T} of \emph{$h_{\scalebox{0.4}{PW}}$} admits an almost K\"ahler structure as in Theorem \emph{\ref{thm:almost}}.
 \end{thm}
 
 \section*{Acknowledgments}
The authors thank the anonymous referee for numerous helpful observations, and in particular for pointing our attention to the co-closed property that led to Corollary \ref{cor:co-closed}.
 
\bibliographystyle{alpha}
\bibliography{almost_Kahler_v2}

\newcommand{\etalchar}[1]{$^{#1}$}
\begin{thebibliography}{CDD{\etalchar{+}}99}

\bibitem[AD03]{apostolov}
Vestislav Apostolov and Tedi Draghici.
\newblock The curvature and the integrability of almost-{K}{\"a}hler manifolds:
  a survey.
\newblock {\em Symplectic and contact topology: interactions and perspectives},
  35:25--53, 2003.

\bibitem[BEE96]{beem}
John~K. Beem, Paul~E. Ehrlich, and Kevin~L. Easley.
\newblock {\em Global {L}orentzian {G}eometry}.
\newblock Marcel Dekker, Inc., $2^{\text{nd}}$ edition, 1996.

\bibitem[Bes07]{besse}
Arthur~L. Besse.
\newblock {\em Einstein manifolds}.
\newblock Springer Science \& Business Media, 2007.

\bibitem[Bla11]{blau2}
Matthias Blau.
\newblock Plane waves and {P}enrose limits.
\newblock {\em Lecture Notes for the ICTP School on Mathematics in String and
  Field Theory (June 2-13 2003)}, 2011.

\bibitem[CDD{\etalchar{+}}99]{catalano}
Domenico Catalano, Filip Defever, Ryszard Deszcz, Marian Hotlo{\'s}, and
  Zbigniew Olszak.
\newblock A note on almost {K}{\"a}hler manifolds.
\newblock In {\em Abhandlungen aus dem Mathematischen Seminar der
  Universit{\"a}t Hamburg}, volume~69, pages 59--65. Springer, 1999.

\bibitem[CFdL85]{cordero}
Luis~A. Cordero, M.~Fernandez, and Manuel de~Le{\'o}n.
\newblock Examples of compact non-{K}{\"a}hler almost {K}{\"a}hler manifolds.
\newblock {\em Proceedings of the American Mathematical Society},
  95(2):280--286, 1985.

\bibitem[CFHP06]{coley}
Alan Coley, Andrea Fuster, Sigbjorn Hervik, and Nicos Pelavas.
\newblock Higher dimensional {VSI} spacetimes.
\newblock {\em Classical and Quantum Gravity}, 23(24):7431, 2006.

\bibitem[CFS03]{candela2003}
A.M. Candela, J.L. Flores, and Miguel S{\'a}nchez.
\newblock On general plane fronted waves. {G}eodesics.
\newblock {\em General Relativity and Gravitation}, 35(4):631--649, 2003.

\bibitem[CS08]{candela}
Anna~Maria Candela and Miguel S{\'a}nchez.
\newblock {\em Geodesics in semi-{R}iemannian manifolds: geometric properties
  and variational tools}, volume~4.
\newblock European Mathematical Society Z{\"u}rich, 2008.

\bibitem[Ebi70]{ebin}
David~G Ebin.
\newblock Completeness of {H}amiltonian vector fields.
\newblock {\em Proceedings of the American Mathematical Society},
  26(4):632--634, 1970.

\bibitem[EIMM86]{Mon}
D.~Eardley, J.~Isenberg, J.~Marsden, and V.~Moncrief.
\newblock Homothetic and conformal symmetries of solutions to {E}instein's
  equations.
\newblock {\em Communications in Mathematical Physics}, 106(1):137--158, 1986.

\bibitem[FS06]{flores}
Jos{\'e}~L. Flores and Miguel S{\'a}nchez.
\newblock On the geometry of pp-wave type spacetimes.
\newblock In {\em Analytical and Numerical Approaches to Mathematical
  Relativity}, pages 79--98. Springer, 2006.

\bibitem[FS20]{FS-Ehlers}
Jos{\'e}~L. Flores and Miguel S{\'a}nchez.
\newblock The {E}hlers-{K}undt conjecture about gravitational waves and
  dynamical systems.
\newblock {\em Journal of Differential Equations}, 268(12):7505--7534, 2020.

\bibitem[Ger69]{gerochL}
Robert Geroch.
\newblock Limits of spacetimes.
\newblock {\em Communications in Mathematical Physics}, 13(3):180--193, 1969.

\bibitem[GL16]{globke}
Wolfgang Globke and Thomas Leistner.
\newblock Locally homogeneous pp-waves.
\newblock {\em Journal of Geometry and Physics}, 108:83--101, 2016.

\bibitem[Jel96]{jelonek}
W{\l}odzimierz Jelonek.
\newblock Some simple examples of almost {K}{\"a}hler non-{K}{\"a}hler
  structures.
\newblock {\em Mathematische Annalen}, 305(1):639--649, 1996.

\bibitem[Kim16]{kim}
Inyoung Kim.
\newblock Almost-{K}{\"a}hler anti-self-dual metrics.
\newblock {\em Annals of Global Analysis and Geometry}, 49(4):369--391, 2016.

\bibitem[Lee18]{Lee}
John~M. Lee.
\newblock {\em Introduction to Riemannian Manifolds}, volume 176.
\newblock Springer, $2^{\text{nd}}$ edition, 2018.

\bibitem[LS16]{leistner}
Thomas Leistner and Daniel Schliebner.
\newblock Completeness of compact {L}orentzian manifolds with abelian holonomy.
\newblock {\em Mathematische Annalen}, 364(3-4):1469--1503, 2016.

\bibitem[Ole14]{olea}
Benjam{\'\i}n Olea.
\newblock Canonical variation of a {L}orentzian metric.
\newblock {\em Journal of Mathematical Analysis and Applications},
  419(1):156--171, 2014.

\bibitem[Pen72]{pen}
Roger Penrose.
\newblock {\em Techniques of {D}ifferential {T}opology in {R}elativity}.
\newblock SIAM, 1972.

\bibitem[Pen76]{penrose}
Roger Penrose.
\newblock Any space-time has a plane wave as a limit.
\newblock In {\em Differential geometry and relativity}, pages 271--275.
  Springer, 1976.

\bibitem[Pet16]{Petersen}
Peter Petersen.
\newblock {\em Riemannian {G}eometry}, volume 171.
\newblock Springer, $3^{\text{rd}}$ edition, 2016.

\bibitem[Phi06]{philip}
Simon Philip.
\newblock Penrose limits of homogeneous spaces.
\newblock {\em Journal of Geometry and Physics}, 56(9):1516--1533, 2006.

\bibitem[SHN{\etalchar{+}}17]{AMS}
Christina Sormani, Denson~C. Hill, Pave{\l} Nurowski, Lydia Bieri, David
  Garfinkle, and Nicol{\'a}s Yunes.
\newblock The {M}athematics of {G}ravitational {W}aves: {A} {T}wo-{P}art
  {F}eature.
\newblock {\em Notices of the AMS}, 64(7):684--707, 2017.

\bibitem[Thu76]{thurston}
W.P. Thurston.
\newblock Some {S}imple {E}xamples of {S}ymplectic {M}anifolds.
\newblock {\em Proceedings of the American Mathematical Society},
  55(2):467--468, 1976.

\bibitem[Wal50]{walker}
A.G. Walker.
\newblock Canonical form for a {R}iemannian space with a parallel field of null
  planes.
\newblock {\em The Quarterly Journal of Mathematics}, 1(1):69--79, 1950.

\bibitem[Wat83]{watson}
Bill Watson.
\newblock New examples of strictly almost {K}{\"a}hler manifolds.
\newblock {\em Proceedings of the American Mathematical Society},
  88(3):541--544, 1983.

\bibitem[WM70]{weinstein}
Alan Weinstein and Jerrold Marsden.
\newblock A comparision theorem for {H}amiltonian vector fields.
\newblock {\em Proceedings of the American Mathematical Society},
  26(4):629--631, 1970.

\end{thebibliography}
\end{document}